\documentclass[10pt]{article}
\usepackage{fullpage}
\usepackage[centertags]{amsmath}
\usepackage{amsfonts}
\usepackage{amssymb}
\usepackage{amsthm}

\usepackage{pdfsync}
\newcommand{\conv}{ {\mathrm{conv}} }
\newcommand{\ext}{ {\mathrm{ext}} }

\theoremstyle{plain}
\newtheorem{theorem}{Theorem}[section]

\newtheorem{proposition}[theorem]{Proposition}

\theoremstyle{definition}
\newtheorem{definition}{Definition}

\title{A Geometric Approach to Combinatorial Fixed-Point Theorems}

\author{Elyot Grant\footnote{Computer Science and Artificial Intelligence Laboratory, Massachusetts Institute of Technology. \texttt{elyot@mit.edu}}, Will Ma\footnote{Operations Research Center, Massachusetts Institute of Technology. \texttt{willma@mit.edu}}}

\begin{document}

\maketitle

\vspace{-20pt}
\begin{abstract}
We develop a geometric framework that unifies several different combinatorial fixed-point theorems related to Tucker's lemma and Sperner's lemma, showing them to be different geometric manifestations of the same topological phenomena. In doing so, we obtain (1) new Tucker-like and Sperner-like fixed-point theorems involving an exponential-sized label set; (2) a generalization of Fan's parity proof of Tucker's Lemma to a much broader class of label sets; and (3) direct proofs of several Sperner-like lemmas from Tucker's lemma via explicit geometric embeddings, without the need for topological fixed-point theorems.  Our work naturally suggests several interesting open questions for future research.
\end{abstract}

\vspace{-10pt}
\section{Introduction}
Combinatorial fixed-point theorems such as the Sperner and Tucker lemmas have generated a wealth of interest in recent decades, in part due to the discovery of important new applications in economics and theoretical computer science (see \cite{Rah12,Yan09,Che09}).  Sperner's lemma is known to be equivalent to the celebrated Brouwer fixed-point theorem, of which it can be regarded as a discrete analogue.  A similar relationship holds between Tucker's lemma and the Borsuk-Ulam theorem---it is easy to show that both are equivalent, with Tucker's lemma effectively serving as a combinatorial version of the topological Borsuk-Ulam theorem (see \cite{Mat03}).

Extensive research has examined the construction of \emph{direct proofs} of the implications among these and other similar theorems (and generalizations), yielding many different proofs of the Sperner and Tucker lemmas via a variety of methods (see \cite{Fre84b, Kuh60, FT81}).  Some of this work has succeeded in connecting fixed-point theorems in the (Brouwer, Sperner)-family to the seemingly unrelated antipodality theorems in the (Borsuk-Ulam, Tucker)-family; for example, Su has shown that it is possible to prove the Brouwer fixed-point theorem directly from the Borsuk-Ulam theorem via an explicit topological construction \cite{Su97}, and \v{Z}ivalcevi\'c \cite{Ziv10} has shown how Ky Fan's \cite{Fan52} generalization of Tucker's lemma implies Sperner's lemma.  However, despite these results suggesting that Tucker's lemma is, in some sense, a ``stronger'' result than Sperner's lemma, the construction of a direct proof that Tucker's lemma implies Sperner's lemma appears to remain an open question \cite{NS12}.

To shed some light on this question, we investigate the Tucker and Sperner lemmas from a \emph{geometric} viewpoint.  Cast in this light, it becomes apparent that the Tucker and Sperner lemmas are actually members of a much larger family of combinatorial fixed-point theorems sharing a common topological structure, but having different geometric manifestations.  Our approach hence unifies many known combinatorial fixed-point theorems, and yields a new Tucker-type lemma and a new Sperner-type lemma, both with an exponential number of labels.  In doing so, we generalize the technique in \cite{Fan52} to obtain a framework that proves our new Tucker-like theorem without any topological fixed-point theorems (\textbf{Section 3}).  As a bonus, our framework also permits us to prove some of the Sperner-like theorems directly from Tucker's lemma via explicit geometric embeddings (\textbf{Section 2}).  Moreover, we derive some insight into why Sperner's lemma may be difficult to prove directly from Tucker's lemma---the analogy between Borsuk-Ulam and Tucker results is geometrically different from the analogy between Brouwer and Sperner results, and alternate Sperner-like theorems provide a more direct analogy.

\

\noindent \emph{Acknowledgements}. We thank Rob Freund for helpful discussions, and referee 2 for several useful comments.  Both authors were partially supported by NSERC PGS-D awards.  The second author was supported in part by NSF grant CCF-1115849 and ONR grants N00014-11-1-0053 and N00014-11-1-0056.

\subsection{Preliminaries}

We begin by briefly outlining our terminology and notation.  For more background, we refer the reader to Matou\v{s}ek \cite{Mat03}.  A set $\sigma \subset \mathbb{R}^n$ is a \emph{$k$-dimensional simplex} (or $k$-simplex) if it is the convex hull $\conv(A)$ of a set $A$ of $k+1$ affinely independent points.  Points in $A$ are termed \emph{vertices} of $\sigma$, denoted $V(\sigma)$.  Simplices of the form $\conv(B)$ where $B \subseteq A$ are known as \emph{faces} of $\sigma$.  We write $e_i$ for the $i^{\mathrm{th}}$ standard basis vector in $\mathbb{R}^n$.
\begin{definition}
For each $n \geq 1$, we define the following subsets of $\mathbb{R}^n$:
\begin{itemize}
\item The \emph{$n$-dimensional unit ball} $B^n$ is defined as $\{x\in R^n, ||x||_2\le 1\}$.
\item The \emph{$(n-1)$-dimensional tetrahedron} $\Delta^{n-1}$ is defined as $\conv\{e_i:1 \leq i \leq n\}$, or equivalently, can be defined as the set of all points $(x_1,\ldots,x_n)$ such that $\sum_{i=1}^n{x_i} = 1$ and $x_i \geq 0$ for all $i$.
\item The \emph{$n$-dimensional octahedron} $\Diamond^n$ is defined as $\conv\{\pm e_i\}$, or equivalently, is the unit ball $\{x\in R^n, ||x||_1\le 1\}$ in the $\ell_1$ norm.  This is also known as the $n$-dimensional \emph{cross-polytope} or \emph{orthoplex}.
\item The \emph{$n$-dimensional cube} $\square^n$ is defined as $\conv\{(x_1, \ldots, x_n):x_i \in \{-1,1\} \}$, or equivalently, is the unit ball $\{x\in R^n, ||x||_\infty\le 1\}$ in the $\ell_\infty$ norm.
\end{itemize}
\end{definition}

Central to all combinatorial fixed-point theorems is the notion of triangulation:
\begin{definition}
If $X \subset \mathbb{R}^n$, a finite family $T$ of simplices is a \emph{triangulation of} $X$ if $\bigcup_{\sigma \in T}\sigma = X$ and the following two properties hold:
\begin{itemize}
\item For each simplex $\sigma \in T$, all faces of $\sigma$ are also simplices in $T$.
\item For any two simplices $\sigma_1,\sigma_2 \in T$, the intersection $\sigma_1 \cap \sigma_2$ is a face of both $\sigma_1$ and $\sigma_2$.
\end{itemize}
\end{definition}
In more general contexts, a \emph{topological triangulation} is defined to only require that $\bigcup_{\sigma \in T}\sigma$ is \emph{homeomorphic to} the set $X$.  This relaxed definition is necessary in order to discuss, e.g.\ finite triangulations of the unit ball.  For brevity, we will omit the word `topological' when the context is clear; when we say \emph{a triangulation of $B^n$}, we mean a triangulation of any set homeomorphic to $B^n$, such as $\Delta^{n}$, $\Diamond^n$, or $\square^n$.  In all other contexts, we will assume that $\bigcup_{\sigma \in T}\sigma = X$ when $T$ is a triangulation of $X$.

\medskip
\noindent
{\bf Combinatorial Fixed-Point Theorems}


We define the \emph{vertices} $V(T)$ of a triangulation $T$ to be the set of all 0-simplices in $T$.  A \emph{label function} $\lambda$ is a mapping from $V(T)$ to a finite \emph{label set} $L$.  In the case of the Tucker and Sperner lemmas, the sets $\{1,\ldots,n+1\}$ or $\{1,-1,\ldots,n,-n\}$ are typically used for $L$.  However, in our paper, we will instead represent these labels as the sets of \emph{extreme points} $\ext(\Delta^{n}) = \{e_1,\ldots,e_{n+1}\}$ and $\ext(\Diamond^{n}) = \{e_1,-e_1\ldots,e_n,-e_n\}$.
Cast in this framework, we shall state Sperner's lemma as follows:
\begin{theorem}[Sperner's Lemma]\label{sperner}
Let $T$ be a triangulation of $\Delta^n$.  Let $\lambda:V(T)\rightarrow \ext(\Delta^n)$ be a label function with the property that for all $x = (x_1,\ldots,x_{n+1}) \in V(T)$, for all $1 \leq i \leq n+1$, if $x_i = 0$, then $\lambda(x)\neq e_i$ (such a $\lambda$ is sometimes called a \emph{proper colouring}).  Then $T$ contains a \emph{panchromatic simplex}---that is, a simplex $\sigma$ such that $\{\lambda(v):v \in V(\sigma)\} = \{e_1,\ldots,e_{n+1}\}$.
\end{theorem}
The key idea in proving Sperner's lemma from the Brouwer fixed-point theorem is to use the labels to construct a mapping from $\Delta^n$ to itself that must contain a fixed point in its interior.  The mapping is designed such that all simplices that are \emph{not} panchromatic are mapped to the boundary of $\Delta^n$, and thus it follows that $T$ must contain a panchromatic simplex.  Nothing about the proof requires any specific properties of $\Delta^n$, so it is natural to consider what happens when we replace $\Delta^n$ with other spaces homeomorphic to it, such as $\Diamond^n$ or $\square^n$.  As it turns out, we can establish results similar to Sperner's lemma in other geometric spaces as long as we have suitable analogies of the notions of \emph{proper colouring} and \emph{panchromatic simplex}.  

The following propositions motivate the alternative notions of \emph{proper colouring} and \emph{panchromatic simplex} that we shall employ when considering different geometric manifestations of Sperner-like theorems:
\begin{proposition} \label{propcomp} Let $\sigma$ be a simplex and let $\lambda: V(\sigma) \rightarrow \ext(\Diamond^n)$ be a label function.  Define a \emph{complementary edge} to be two vertices $v_1, v_2 \in V(\sigma)$ with $\lambda(v_1) = -\lambda(v_2)$.  Then $\conv\{\lambda(v):v\in V(\sigma)\}$ intersects the interior of $\Diamond^n$ if and only if $\sigma$ contains a complementary edge.
\end{proposition}
\begin{proposition} \label{propneutral}
Let $\sigma$ be a simplex and let $\lambda: V(\sigma) \rightarrow \ext(\square^n)$ be a label function.  We say $\sigma$ is a \emph{neutral simplex} if for all $1 \leq i \leq n$, there exist vertices $v_1, v_2 \in  V(\sigma)$ such that $\lambda_i(v_1)=-1$, $\lambda_i(v_2)=+1$, where $\lambda_i(v)$ is the $i^{\mathrm{th}}$ coordinate of $\lambda(v)$.  Then $\conv\{\lambda(v):v\in V(\sigma)\}$ intersects the interior of $\square^n$ if and only if $\sigma$ is a neutral simplex.
\end{proposition}
As a slight abuse of definitions, we shall say that $T$ has a complementary edge whenever some simplex in $T$ does.

Two Sperner-like theorems can immediately be derived by imposing the right labelling constraints:
\begin{theorem}[Octahedral Sperner with Octahedral Labels]\label{thspoctoct}
Let $T$ be a triangulation of $\Diamond^n$.  Let $\lambda:V(T)\rightarrow \ext(\Diamond^n)$ be a label function such that for all boundary vertices $x = (x_1,\ldots,x_n) \in V(T) \cap \partial(\Diamond^n)$, for all $1 \leq i \leq n$, if $x_i \ge 0$ (respectively, if $x_i \le 0$), then $\lambda(x)\neq -e_i$ (respectively, $\lambda(x)\neq e_i$).
Then $T$ contains a complementary edge.
\end{theorem}
\begin{theorem}[Cubical Sperner with Cubical Labels]\label{thspcubcub}
Let $T$ be a triangulation of $\square^n$.  Let $\lambda:V(T)\rightarrow \ext(\square^n)$ be a label function such that for all vertices $x = (x_1,\ldots,x_n) \in V(T)$, for all $1 \leq i \leq n$, if $x_i \in \{-1,1\}$, then $\lambda(x)_i = x_i$.  Then $T$ contains a neutral simplex.
\end{theorem}

Theorem~\ref{thspoctoct} is a special case of a theorem originally conjectured by Atanassov and proven by De Loera et al. \cite{Ata96,LPS02}.  Theorem~\ref{thspcubcub} is implied by a theorem of Kuhn \cite{Kuh60}.  However, Kuhn actually relabels the vertices of the triangulation asymmetrically using $(n+1)$ labels, and demonstrates the existence of a panchromatic simplex under the relabelling, which directly implies the existence of a neutral simplex under the original labelling.  His result can be regarded as a ``Cubical Sperner with Tetrahedral Labels'', where the \emph{domain} over which the triangulation takes place is distinct from the \emph{codomain} whose extreme points are the labels (though the domain and codomain are still homeomorphic).  As it turns out, allowing the domain and codomain to differ yields several additional fixed-point theorems.  We give some examples:
\begin{theorem}[Cubical Sperner with Octahedral Labels]\label{thspcuboct}
Let $T$ be a triangulation of $\square^n$.  Let $\lambda:V(T)\rightarrow \ext(\Diamond^n)$ be a label function with the property that for all vertices $x = (x_1,\ldots,x_n) \in V(T)$, for all $1 \leq i \leq n$, if $x_i\in\{-1,1\}$, then $\lambda_i(x)\neq -x_i e_i$.  Then $T$ contains a complementary edge.
\end{theorem}
\begin{theorem}[Octahedral Sperner with Cubical Labels]\label{thspoctcub}
Let $T$ be a triangulation of $\Diamond^n$.  Let $\lambda:V(T)\rightarrow \ext(\square^n)$ be a label function with the property that for all vertices $x \in V(T)$, for all $v\in\ext(\square^n)$, if $v^T x=1$, then $\lambda(x)\neq-v$.  Then $T$ contains a neutral simplex.
\end{theorem}
Theorem~\ref{thspcuboct} is due to Freund \cite{Fre84b}.  Theorem~\ref{thspoctcub} appears to be novel, though it is related to a result of Freund \cite{Fre89} that assumes stronger conditions on the label function, and establishes a stronger result---namely, that $T$ contains a simplex whose labels, when regarded as a subset of the extreme points of $\square^n$, contain the origin (or some other specific point) in their convex hull.  The property of having such a simplex is weaker than the property of having a complementary edge (since any simplex $\sigma$ containing a complementary edge must contain two vertices whose labels have the origin as their midpoint, and hence the convex hull of the labels used must contain the origin).  However, it is stronger than the property of having a neutral simplex; for example, the set $S = \{(1,1,-1),(1,-1,1),(-1,1,1),(1,1,1)\}$ does not contain the origin in its convex hull, but corresponds to the labels of a neutral simplex.  One may observe that a triangulation of $\Diamond^3$ can be labelled using labels in the set $S$ while respecting the conditions of Theorem~\ref{thspoctcub}, so our result cannot be implied by Freund's.

It is not clear how one might prove Theorem~\ref{thspoctcub} combinatorially.  It does not seem possible to employ a relabelling trick similar to Kuhn's to reduce Theorem~\ref{thspoctcub} to a statement involving only polynomially many different labels.  Additionally, the methods of Freund do not appear to be applicable.  In Section 4, we provide a proof of Theorem~\ref{thspoctcub} via the Brouwer fixed-point theorem.  We leave it as an open question to prove this theorem via combinatorial means.

It is undoubtedly possible to derive further Sperner-like theorems using other domains and codomains, but the aforementioned results are the most natural due to the symmetrical geometry of the cube and octahedron.  Octahedral labels are particularly nice due to their relation to Tucker's lemma; indeed, it is possible to prove Theorem~\ref{thspcuboct} and Theorem~\ref{thspoctoct} directly from Tucker's lemma, as we show in Section 2.

To state Tucker-like fixed-point theorems, we require a notion of antipodality:
\begin{definition}
Let $T$ be a topological triangulation of $B^n$ with $X=\bigcup_{\sigma \in T}\sigma$.  We define the \emph{boundary} of $T$ to be the triangulation $\partial(T)\subset T$ containing all simplices in $T$ that lie entirely in the boundary $\partial(X)$.  $T$ is said to be \emph{antipodally symmetric on the boundary} if, for all simplices $\sigma \in \partial(T)$, the reflected simplex $-\sigma$ also lies in $\partial(T)$.
\end{definition}

Tucker's lemma can then be stated as follows (originally in \cite{Tuc46}; see \cite{Meu06} for a survey):
\begin{theorem}[Tucker's Lemma]\label{tucker}
Let $T$ be a triangulation of $B^n$ that is antipodally symmetric on the boundary of the domain $X$.  Let $\lambda:V(T)\rightarrow \ext(\Diamond^n)$ be a label function such that $\lambda(v) = -\lambda(-v)$ for all $v \in \partial(X)$.  Then $T$ contains a complementary edge.
\end{theorem}
Our geometric intuition suggests that similar theorems should be possible using labels from other codomains homeomorphic to $B^n$.  Although it does not make sense to use labels in $\ext(\Delta^n)$ due to there being no natural notion of \emph{negation} among them, we can use labels in $\ext(\square^n)$ to obtain the following new result: 

\begin{theorem}[Tucker's Lemma with Cubical Labels]\label{thtucub}
Let $T$ be a triangulation of $B^n$ that is antipodally symmetric on the boundary of the domain $X$.  Let $\lambda:V(T)\rightarrow \ext(\square^n)$ be a label function such that $\lambda(v) = -\lambda(-v)$ for all $v \in \partial(X)$.  Then $T$ contains a neutral simplex.
\end{theorem}
As one might expect, it is straightforward to prove Theorem~\ref{thtucub} topologically using the Borsuk-Ulam theorem.  However, we can establish this theorem combinatorially by extending the framework of Ky Fan \cite{Fan52} to a broader class of possible labellings of triangulations.  In Section~3, we state and prove this generalization of Fan's theorem, and show how it implies Theorem~\ref{thtucub}.




\section{Geometric Proofs of Sperner-like Theorems}
In this section, we describe a technique that enables us to explicitly construct geometric reductions between combinatorial fixed-point theorems.  We illustrate our technique through an example:
\begin{theorem} \label{dilligaf}
Tucker's Theorem implies Theorem~\ref{thspoctoct} (Octahedral Sperner with Octahedral Labels).
\end{theorem}
\begin{proof}
Let $T$ be a triangulation of $\Diamond^n$ with label function $\lambda:V(T)\rightarrow \ext(\Diamond^n)$ satisfying the conditions of Theorem~\ref{thspoctoct}.  Let $X = 2\Diamond^n$ be a dilated copy of the $n$-dimensional octahedron, so that $\Diamond^n$ lies entirely within the interior of $X$.  Our key idea is to extend the triangulation $T$ and label function $\lambda$ to a triangulation $T^*$ of $X$ and a label function $\lambda^*: V(T^*) \rightarrow \ext(\Diamond^n)$ so that the following properties hold:
\begin{enumerate}
\item $T \subset T^*$, and $\lambda^*(v) = \lambda(v)$ for each vertex $v$ in $V(T)$.
\item $T^*$ is antipodally symmetric on the boundary, and $\lambda^*(v) = -\lambda^*(-v)$ for each $v \in V(T^*) \cap \partial(X)$.
\item There are no complementary edges in $T^* \setminus T$.
\end{enumerate}
If we can construct such a $T^*$ and $\lambda^*$, then Theorem~\ref{thspoctoct} immediately follows from Tucker's lemma, since $T^*$ must contain a complementary edge if property (2) is true, and this complementary edge must then lie in $T$ by properties (1) and (3).

It remains to define a new triangulation $T^*$ and labelling $\lambda^*$.  Let $Y = (X \setminus \Diamond^n) \cup \partial(\Diamond^n)$.  To construct $T^*$, we shall first construct a triangulation of $Y$ that agrees with $T$ on $\partial(\Diamond^n)$, and contains no vertices lying on $\partial(X)$ other than those in $\ext(X)$.  This construction is made significantly simpler through the application of known theory regarding triangulations of polytopes; the reader is referred to \cite{Lee97} for background.  We first subdivide the region $Y$ into $2^n$ pieces via slicing by the $n$ different $(n-1)$-dimensional coordinate hyperplanes.  Each piece that remains is a convex $n$-polytope lying entirely within a single orthant.  We consider the subdivisions obtained by adjoining each polytope to all of the simplices in $T$ that lie on the respective facet of $\Diamond^n$.  We then refine each of these subdivisions to a triangulation using sequential `pushing' or `pulling' steps as described in \cite{Lee97}.  Note that we must perform these refinements iteratively, one subdivision at a time, in order to ensure that simplicies lying in the intersection of two neighbouring subdivisions are respected in the triangulation.  

After triangulating $Y$, we extend this triangulation to all of $X$ by simply adding the simplices in our original triangulation $T$ of $\Diamond^n$, noting that the resulting collection of simplices is indeed a triangulation, because the simplices agree at the boundard of $\Diamond^n$.  We define  $T^*$ to be the resulting triangulation obtained by adding the triangulation of $Y$ to $T$.

We next define an extended labelling $\lambda^*: V(T^*) \rightarrow \ext(\Diamond^n)$.  We set $\lambda^*(v) = \lambda(v)$ for each vertex $v \in V(T)$, and $\lambda^*(2e_i) = e_i$ for the remaining vertices of $T^*$---the extreme points of $X$.  Because we split the region $Y$ into orthants before triangulating it, we can be sure that no complementary edge exists among the simplices in $T^* \setminus T$, as there are no edges between in $T^* \setminus T$ between vertices lying within different orthants.  However, $T^*$ is antipodally symmetric on the boundary of $X$, and $\lambda^*(v) = -\lambda^*(-v)$ for each $v \in V(T^*) \cap \partial(X)$, so Tucker's theorem implies that $T^*$ must contain a complementary edge somewhere.  It therefore must lie in $T$, completing the proof.
\end{proof}

We can use the techniques of Theorem~\ref{dilligaf} to prove Sperner-like combinatorial fixed-point theorems directly via implications from other fixed-point theorems.  The proofs all involve similar techniques to those used above in the proof of Theorem~\ref{dilligaf}; the key steps in proving that fixed point theorem $A$ implies Sperner-like theorem $B$ always follows the same pattern:
\begin{enumerate}
\item First, we embed the domain $Z$ of theorem $B$ inside the interior of the domain $X$ in used in theorem $A$.
\item We construct an extension of the triangulation $T$ from theorem $B$ to a triangulation $T^*$ of $X$ by constructing a triangulation of $X \setminus Z \cup \partial(Z)$ that agrees with $T$ on $\partial(Z)$.  We ensure that there are no vertices lying on $\partial(X)$ other than those in $\ext(X)$.  We also ensure that this triangulation has no edges that go between two different orthants (in other words, all edges that go between a point in orthant $O_1$ and a point in orthant $O_2$ have an end in $O_1 \cap O_2$.)  For technical reasons (chiefly, lack of convexity), this may require adding additional ``Steiner'' vertices in the interior of $X$, but this poses no problem.
\item We extend the labelling of $T$, as given in theorem $B$, to a labelling of all the vertices used in our extended triangulation $T^*$.  We ensure that this labelling has two important properties: First, on the boundary $\partial(X)$, it must respect the labelling constraints required by theorem $A$ (either an antipodality condition if theorem $A$ is a Tucker-like theorem, or a \emph{proper labelling condition} if theorem $A$ is a Sperner-like theorem).  Secondly, it must add no ``forbidden'' simplices, with respect to the constraints imposed in theorem $B$.
\item Finally, we argue that because theorem $A$ implies the existence of a forbidden simplex in $T^*$, there must exist a forbidden simplex in $T$, proving theorem $B$.
\end{enumerate}
Using this idea, alongside standard triangulation techniques, it is straightforward to establish all of the following:
\begin{theorem}
Tucker's Theorem implies Theorem~\ref{thspcuboct} (Cubical Sperner with Octahedral Labels).
\end{theorem}
\begin{theorem}
Theorem~\ref{thspcuboct} (Cubical Sperner with Octahedral Labels) implies Theorem~\ref{thspoctoct} (Octahedral Sperner with Octahedral Labels).
\end{theorem}
\begin{theorem}
Theorem~\ref{thspoctoct} (Octahedral Sperner with Octahedral Labels) implies Theorem~\ref{thspcuboct} (Cubical Sperner with Octahedral Labels).
\end{theorem}
The proofs of these theorems are all similar to the proof of Theorem~\ref{dilligaf}.  Technical details on how to triangulate non-convex regions, such as an orthant of an origin-centered octohedron with an orthant of an origin-centered cube removed, can be found in \cite{Lee97}.

Unfortunately, the style of geometric argument we describe here relies crucially on a labelling scheme in which negations are permitted.  Accordingly, it appears that additional insight is required in order for it to be possible to use a geometric construction of this nature to directly prove Sperner's lemma from a Tucker-like theorem.

\section{Generalized Parity Framework for Tucker-like Theorems}
In this section, we give a combinatorial proof of Theorem~\ref{thtucub} via a generalization of the inductive parity proof technique of Fan \cite{Fan52}.  We assume our triangulation $T$ is \emph{aligned with hemispheres}, a standard assumption in proofs of Tucker-like lemmas \cite{Mat03,SS03,PS05}:
\begin{definition}
Let $T$ be a triangulation of $B^n$ that is antipodally symmetric on the boundary.  $T$ is \emph{aligned with hemispheres} if there exists a sequence $T^n=T,T^{n-1},\ldots,T^0$, where $T^i$ is a triangulation of $B^i$ that is antipodally symmetric on the boundary, such that $T^{i-1}\cup(-T^{i-1})=\partial(T^{i})$ and $T^{i-1}\cap(-T^{i-1})=\partial(T^{i-1})$.
\end{definition}
Roughly speaking, this means that $\partial(T)$ (a triangulation of $\partial(B^{n})$) can be decomposed into two antipodal triangulations of $B^{n-1}$ that intersect along an equatorial triangulation of $\partial(B^{n-1})$.  Repeating this process recursively shall later facilitate an inductive argument.

We proceed with further definitions.  Suppose we have a label set that can be partitioned into pairs of \emph{opposite} labels.  We say such a label set is \emph{strictly symmetric}.  Define a \emph{labelling} $\ell$ to be an unordered, unoriented way of labelling a simplex, i.e.\ a multi-set of labels.  There is a natural notion of an \emph{opposite labelling} $-\ell$, where we take the opposite of each label in the multi-set.  We use \emph{$i$-labelling} to refer to the labelling of an $i$-simplex, i.e.\ a multi-set of $i+1$ labels.  Let $L^i$ denote the set of $i$-labellings.  Note that $L^i$ may not be strictly symmetric even though the label set is strictly symmetric, e.g.\ labelling $\{+e_1,-e_1\}$ is its own opposite.

Now we define the set of \emph{forbidden labellings} $F\subseteq\bigcup_{i=0}^n L^i$.  (Assume $F\cap L^0=\emptyset$; otherwise we would just make the label set smaller.)  In the case of Tucker's lemma, $F$ is the set of labellings with a complementary edge.  In the case of Theorem~\ref{thtucub}, $F$ is the set of labellings with a neutral simplex.  The idea is to choose $F$'s such that $L^i\backslash F$ is strictly symmetric for all $0\le i\le n$.  That is, $F$ must be symmetric and contain all the labellings which are their own opposite.

If $M^i$ is a strictly symmetric subset of $i$-labellings, then from each pair of opposite labellings, we can choose one to be the `+' labelling and the other to be the `-' labelling.  Let $M^i_+,M^i_-$ denote the set of $+,-$ labellings, respectively.  We call $M^i_+\cup M^i_-$ a \emph{partitioning} of $M^i$.

\begin{theorem}[Parity Proof Framework]
\label{parity_proof_framework}
Let $T$ be a triangulation of $B^n$ that is antipodally symmetric on the boundary and aligned with hemispheres.  Fix the label set, and let $F$ be any set of forbidden labellings such that $L^i\backslash F$ is strictly symmetric for all $i=0,\ldots,n$.  Let $\lambda$ be a label function on $V(T)$ such that $\lambda(v) = -\lambda(-v)$ for all $v$ on the boundary.  Suppose $T$ contains no simplices with a labelling from $F$.  Define $M^0:=L^0$.  Then for $i=1,\ldots,n$:
\begin{itemize}
\item $M^{i-1}$ is strictly symmetric, so we can partition it into $M^{i-1}_+\cup M^{i-1}_-$
\item Define $M^i$ to be the set of $\ell^i$ in $L^i\backslash F$ that contain an odd number of $(i-1)$-labellings from $M^{i-1}_+$.  Then $T^i$ contains an odd number of $i$-simplices with labellings from $M^i$. 
\end{itemize}
\end{theorem}
\begin{proof}
First, we need some notation.  Assume $M^i_+\cup M^i_-$ is chosen, and $j\ge i$.
\begin{itemize}
\item For $j$-labelling $\ell^j$, let $deg^i(\ell^j)$ denote the number of $\ell^i\subseteq\ell^j$ such that $\ell^i\in M^i$.
\item If $j=i$, then $deg^i(\ell^j)$ is simply the indicator function for whether $\ell^j\in M^i$.
\item For $j$-simplex $\sigma_j$, let $deg^i(\sigma_j)$ denote the number of $i$-dimensional faces with a labelling in $M^i$.
\item Similarly define $deg^i_+(\ell^j)$, $deg^i_+(\sigma_j)$ for $M^i_+$, and $deg^i_-(\ell^j)$, $deg^i_-(\sigma_j)$ for $M^i_-$.
\end{itemize}

Suppose $i>0$.  To show $M^i$ is strictly symmetric, we first show that every $i$-labelling $\ell^i\in L^i\backslash F$ contains an even number of $(i-1)$-labellings from $M^{i-1}$.
\begin{eqnarray*}
deg^{i-1}(\ell^i) & \equiv & \sum_{\ell^{i-1}\subset\ell^i}deg^{i-1}(\ell^{i-1}) \\
& \equiv & \sum_{\ell^{i-1}\subset\ell^i}deg^{i-2}_+(\ell^{i-1}) \\
& \equiv & \sum_{\ell^{i-1}\subset\ell^i}\sum_{\ell^{i-2}\subset\ell^{i-1}}deg^{i-2}_+(\ell^{i-2}) \\
& \equiv & \sum_{\ell^{i-2}\subset\ell^i}2\cdot deg^{i-2}_+(\ell^{i-2}) \\
& \equiv & 0\mod{2} \\
\end{eqnarray*}
where all the equivalences follow from definitions except the fourth one, which follows from the fact that each $\ell^{i-2}\subset\ell^i$ is contained in exactly $2$ of the $\ell^{i-1}$'s.  (This proof doesn't make sense for $i=1$, but that case is trivial.)

Now, suppose $\ell^i\in M^i$.  $deg^{i-1}_+(\ell^i)$ is odd, so $deg^{i-1}_-(\ell^i)$ is also odd, since $deg^{i-1}(\ell^i)$ is even.  But $deg^{i-1}_-(\ell^i)=deg^{i-1}_+(-\ell^i)$, since $M^{i-1}$ is strictly symmetric.  Hence $deg^{i-1}_+(-\ell^i)$ is odd and $-\ell^i\in M^i$, so $M_i$ is symmetric.  Furthermore, $M_i$ is strictly symmetric since it is a subset of $L^i\backslash F$, which is strictly symmetric.

To show that $T^i$ contains an odd number of $i$-simplices with labellings from $M^i$:
\begin{eqnarray*}
\sum_{\sigma_i\in T^i}deg^i(\sigma_i) & \equiv & \sum_{\sigma_i\in T^i}deg^{i-1}_+(\sigma_i) \\
& \equiv & \sum_{\sigma_i\in T^i}\sum_{\sigma_{i-1}\subset\sigma_i}deg^{i-1}_+(\sigma_{i-1}) \\
& \equiv & \sum_{\sigma_{i-1}\in\partial T^i}deg^{i-1}_+(\sigma_{i-1})+\sum_{\sigma_{i-1}\in T^i\backslash\partial T^i}2\cdot deg^{i-1}_+(\sigma_{i-1}) \\
& \equiv & \sum_{\sigma_{i-1}\in T^{i-1}\cup-T^{i-1}}deg^{i-1}_+(\sigma_{i-1}) \\
& \equiv & \sum_{\sigma_{i-1}\in T^{i-1}}deg^{i-1}(\sigma_{i-1}) \\
& \equiv & 1\mod{2} \\
\end{eqnarray*}
where
\begin{itemize}
\item the third equivalence follows from the fact that each $\sigma_{i-1}$ on the boundary of $T^i$ is contained in exactly $1$ of the $\sigma_i$'s, while every $\sigma_{i-1}$ on the interior of $T^i$ is contained in exactly $2$ of the $\sigma_i$s;
\item the fourth equivalence follows from the hemispheres aligned with $T$;
\item the fifth equivalence follows from the fact that $\lambda$ is antipodal on the boundary;
\item the sixth equivalence follows from the inductive hypothesis.
\end{itemize}
This completes the induction and the proof.
\end{proof}
One may observe that we provided an elementary inductive proof of Theorem~\ref{parity_proof_framework}, similar in spirit to Fan's original combinatorial proof of his theorem.  However, it is also possible to establish this theorem using algebraic methods, following a scheme such as that used by Meunier\cite{Meu06a}.

It is not hard to see that, with the right parameters, this framework can be used to immediately establish Tucker's lemma:
\begin{theorem}\label{paritytucker}
Theorem~\ref{parity_proof_framework} directly implies Tucker's lemma.
\end{theorem}
\begin{proof}
Theorem~\ref{parity_proof_framework} is true for any sequence of partitionings.  To prove Tucker's lemma, we must produce a sequence of partitionings such that $M^n=\emptyset$, contradicting the existence of a triangulation $T$ that doesn't contain a forbidden labelling.

Let $F$ be the set of labellings that contain a complementary edge; note that $L^i\backslash F$ is indeed symmetric for all $i$.  Suppose we had a triangulation $T$ that was a counterexample to Tucker's Lemma; $T$ satisfies the conditions of Theorem~\ref{parity_proof_framework}.  At each stage, use the following rule to partition $M^{i-1}$: given a labelling $\ell^{i-1}$, let $j$ be the smallest positive integer such that either $+j$ or $-j$ appears in $\ell^{i-1}$.  Since $\ell^{i-1}$ cannot contain a complementary edge, we cannot have both appear.  If $+j$ appears, then choose $\ell^{i-1}$ to be a `+' labelling; otherwise choose it to be `-'.  Note that this is a legal partitioning of $M^{i-1}$.  Now, for $i=1,\ldots,n$, it can be inductively observed that:
\begin{itemize}
\item $M^{i-1}_+=\{\{k_1,-k_2,k_3,\ldots,(-1)^{i-1}k_{i}\}:1\le k_1<\ldots<k_{i}\le n\}$
\item $M^{i-1}_-=\{\{-k_1,+k_2,-k_3,\ldots,(-1)^{i}k_{i}\}:1\le k_1<\ldots<k_{i}\le n\}$
\item $M^i=\{\{k_1,-k_2,k_3,\ldots,(-1)^{i}k_{i+1}\},\{-k_1,+k_2,-k_3,\ldots,(-1)^{i+1}k_{i+1}\}:1\le k_1<\ldots<k_{i+1}\le m\}$
\end{itemize}
Thus $M^n=\emptyset$, completing the proof.
\end{proof}

Theorem~\ref{parity_proof_framework} also yields a purely combinatorial proof of Theorem~\ref{thtucub}.
\begin{theorem}
Theorem~\ref{parity_proof_framework} directly implies Theorem~\ref{thtucub} [Tucker's lemma with cubical labels].
\end{theorem}
\begin{proof}
Let $F$ be the set of labellings that contain a neutral simplex.  Suppose we had a triangulation $T$ that was a counterexample to Theorem~\ref{thtucub}.  At each stage, use the following rule to partition $M^{i-1}$: given a labelling $\ell^{i-1}$, let $\Phi(\ell^{i-1})$ be the smallest coordinate on which all $i$ labels of $\ell^{i-1}$ agree.  Since $\ell^{i-1}$ cannot be a neutral simplex, this coordinate must exist.  If the labels are all positive on this coordinate, then choose $\ell^{i-1}$ to be a `+' labelling; otherwise choose it to be `-'.

For $i=0,\ldots,n$, we will inductively show that for all $\ell^i\in M^i$, $\Phi(\ell^i)\ge i+1$.  The base case $i=0$ is trivial.  Suppose $i>0$ and $\ell^i\in M^i$.  By definition, $\ell^i$ contains a labelling $\ell^{i-1}\in M^{i-1}_+$.  The induction hypothesis says $\Phi(\ell^{i-1})\ge i$.  Thus, since $\ell^i$ contains $\ell^{i-1}$, $\ell^i$ cannot have all labels agree until at least the $i$'th coordinate, and if they agree on the $i$'th coordinate, it must be on `+'.

Now, $M^i$ is strictly symmetric, so $-\ell^i\in M^i$, which means $-\ell^i$ contains an odd number of $(i-1)$-labellings from $M^{i-1}_+$.  But then $\ell^i$ contains an odd number of, i.e.\ at least one $(i-1)$-labellings from $M^{i-1}_-$, so we can draw a corresponding conclusion: $\ell^i$ cannot have all labels agree until at least the $i$'th coordinate, and if they agree on the $i$'th coordinate, it must be on `-'.

Combining these two conclusions, $\ell^i$ cannot have all labels agree on the $i$'th coordinate, so we get $\Phi(\ell^i)\ge i+1$, completing the induction.  When $i=n$, for all $\ell^n\in M^n$, $\Phi(\ell^n)\ge n+1$.  But this is impossible, so $M^n=\emptyset$, contradicting the fact that $T^n=T$ contains an odd number of labellings from $M^n$.
\end{proof}
Note that our framework allowed us to reach the conclusion without being able to characterize $M^i$ at each stage.  We leave it as an open question to come up with a combinatorial characterization of $M^i$.

\section{Topological Proof of Theorem~\ref{thspoctcub}}
Here, we provide a straightforward proof of Theorem~\ref{thspoctcub} (Octahedral Sperner with Cubical Labels) via the Brouwer fixed-point theorem.  We follow the same standard technique that is typically used in the proof of Sperner's Lemma \cite{Mat03}.
\newenvironment{appendixtheorem}[2][Theorem]{\begin{trivlist}\item[\hskip \labelsep {\bfseries #1}\hskip \labelsep {\bfseries #2}]}{\end{trivlist}}
\begin{appendixtheorem}{\ref{thspoctcub}}[Octahedral Sperner with Cubical Labels]
Let $T$ be a triangulation of $\Diamond^n$.  Let $\lambda:V(T)\rightarrow \ext(\square^n)$ be a label function with the property that for all vertices $x \in V(T)$, for all $v\in\ext(\square^n)$, if $v^T x=1$, then $\lambda(x)\neq-v$.  Then $T$ contains a neutral simplex.
\end{appendixtheorem}
\begin{proof}
Suppose for contradiction we had a $T$, $\lambda$ that was a counterexample to Theorem~\ref{thspoctcub}. Our goal is to construct a continuous function from $\Diamond^n$ to itself that has no fixed points, contradicting Brouwer's fixed-point theorem. Let $f:\Diamond^n\rightarrow(n\cdot\square^n)$ be the natural linear extension of $\lambda$ to the entire octahedron, except we dilate the codomain by a factor of $n$ (and instead of mapping to $(1,\ldots,1)$, we map to $(n,\ldots,n)$, etc.). We will want to consider $-f(x)$.

Now, fix an $x\in\partial\Diamond^n$. Define $C(x)=\{v\in\ext(\square^n):v^T x=1\}$. Since $x\in\partial\Diamond^n$, $C(x)\neq\emptyset$. We want to define the region of $\partial\Diamond^n$ to which $x$ belongs by a vector $S(x)\in\{-1,0,1\}^n$. $\forall i\in[n]$, define
\begin{itemize}
\item $S_i(x)=1$ if $\forall v\in C(x)$, $v_i=+1$
\item $S_i(x)=-1$ if $\forall v\in C(x)$, $v_i=-1$
\item $S_i(x)=0$ otherwise
\end{itemize}
There are $3^n-1$ different regions excluding the interior. The key geometric observation about $\Diamond^n$ is that if we know $S(x)$, then $C(x)$ is guaranteed to be the maximum of all subsets of $\ext(\square^n)$ that give rise to $S(x)$.

Now we need to define corresponding regions on $\partial(n\cdot\square^n)$. For $s=S(x)$, we want to define $R(s)$, the region of $\partial(n\cdot\square^n)$ corresponding to $s$. Note that these regions will overlap. Say $x\in R(s)$ if and only if for all $i\in[n]$:
\begin{itemize}
\item If $s_i=1$, then $x_i\ge n-1$
\item If $s_i=-1$, then $x_i\le -(n-1)$
\item If $s_i=0$, then $-(n-1)\le x_i\le n-1$
\end{itemize}

Now, take an $x\in\partial\Diamond^n$ and consider $s:S(x)$. Let $W(s)=\{t\in\ext(\square^n):t_i=s_i\ if\ |s_i|=1\}$. $-f(x)$ is a convex combination of at most $n-1$ vertices of $n\cdot\square^n$. Let's consider what these vertices can be. By the observation that $C(x)$ is maximal, none of these vertices can be those in $n\cdot W(s)$. Thus there exists some coordinate $i$ such that $|s_i|=1$ and at least one of these vertices ``disagrees'' with $s_i$. We conclude that
\begin{itemize}
\item If $s_i=1$, then $-f(x)\le\frac{n-1}{n}n+\frac{1}{n}(-n)<n-1$.
\item If $s_i=-1$, then $-f(x)\ge\frac{n-1}{n}(-n)+\frac{1}{n}(n)>-(n-1)$.
\end{itemize}
Thus, $-f(x)\notin R(x)$.

Now, let $g(x)$ be the continuous deformation that maps $n\cdot\square^n$ to $\Diamond^n$, and bijectively maps $R(x)$ to $R'(x):=\bigcap_{v\in C(x)}\{y:v^T y=1\}$ for all points $x\in\partial(\Diamond^n)$. Note that all the boundaries line up geometrically. This map is continuous, and this map can be continuously extended to the interiors.

Now consider the composition $g(-f(x)):\Diamond^n\rightarrow\Diamond^n$. Since there is no neutral simplex, $f(x)$ maps nothing to the interior and hence there can be no fixed points on the interior. If $x\in\partial\Diamond^n$, then $x\in R'(x)$, but $-f(x)\notin R(x)$. Then $g(-f(x))\notin R'(x)$, so $x$ cannot be a fixed point. We've constructed a continuous function from $\Diamond^n$ to itself with no fixed points, contradicting Brouwer's fixed-point theorem and completing the proof.
\end{proof}

{
\bibliographystyle{amsalpha}
\bibliography{topology}

\providecommand{\bysame}{\leavevmode\hbox to3em{\hrulefill}\thinspace}
\providecommand{\MR}{\relax\ifhmode\unskip\space\fi MR }
\providecommand{\MRhref}[2]{%
  \href{http://www.ams.org/mathscinet-getitem?mr=#1}{#2}
}
\providecommand{\href}[2]{#2}
\begin{thebibliography}{Meu06b}

\bibitem[Ata96]{Ata96}
K.~T. Atanassov, \emph{On {S}perner's lemma}, Studia Sci. Math. Hungar.
  \textbf{32} (1996), no.~1-2, 71--74.

\bibitem[CDT09]{Che09}
X.~Chen, X.~Deng, and S-H. Teng, \emph{Settling the complexity of computing
  two-player nash equilibria}, Journal of the ACM (JACM) \textbf{56} (2009),
  no.~3, 14.

\bibitem[Fan52]{Fan52}
K.~Fan, \emph{A generalization of {T}ucker's combinatorial lemma with
  topological applications}, Ann. of Math. (2) \textbf{56} (1952), 431--437.

\bibitem[Fre84]{Fre84b}
R.~M. Freund, \emph{Variable dimension complexes. {II}. {A} unified approach to
  some combinatorial lemmas in topology}, Math. Oper. Res. \textbf{9} (1984),
  no.~4, 498--509.

\bibitem[Fre89]{Fre89}
\bysame, \emph{Combinatorial analogs of brouwer's fixed-point theorem on a
  bounded polyhedron}, JCTB \textbf{47} (1989), no.~2, 192--219.

\bibitem[FT81]{FT81}
R.~M. Freund and M.~J. Todd, \emph{A constructive proof of {T}ucker's
  combinatorial lemma}, J. Combin. Theory Ser. A \textbf{30} (1981), no.~3,
  321--325.

\bibitem[Kuh60]{Kuh60}
H.~W. Kuhn, \emph{Some combinatorial lemmas in topology}, IBM J. Res. Develop.
  \textbf{4} (1960), 508--524.

\bibitem[Lee97]{Lee97}
C.~W. Lee, \emph{Subdivisions and triangulations of polytopes}, Handbook of
  discrete and computational geometry, CRC Press Ser. Discrete Math. Appl.,
  CRC, Boca Raton, FL, 1997, pp.~271--290.

\bibitem[LPS02]{LPS02}
J.~A.~De Loera, E.~Peterson, and F.~E. Su, \emph{A polytopal generalization of
  {S}perner's lemma}, J. Combin. Theory Ser. A \textbf{100} (2002), no.~1,
  1--26.

\bibitem[Mat03]{Mat03}
J.~Matou{\v{s}}ek, \emph{Using the {B}orsuk-{U}lam theorem}, Springer-Verlag,
  Berlin, 2003.

\bibitem[Meu06a]{Meu06}
F.~Meunier, \emph{Pleins etiquetages et configurations equilibrees : aspects
  topologiques de l'optimisation combinatoire}, Thesis (French) (2006).

\bibitem[Meu06b]{Meu06a}
Fr{\'e}d{\'e}ric Meunier, \emph{A zq-fan theorem}, Les Cahiers du Laboratoire
  Leibniz (2006).

\bibitem[NS12]{NS12}
K.~L. Nyman and F.~E. Su, \emph{A borsuk-ulam equivalent that directly implies
  sperner's lemma}, Unpublished Manuscript (2012).

\bibitem[PS05]{PS05}
T.~Prescott and F.~E. Su, \emph{A constructive proof of {K}y {F}an's
  generalization of {T}ucker's lemma}, J. Combin. Theory Ser. A \textbf{111}
  (2005), no.~2, 257--265.

\bibitem[Rah12]{Rah12}
M.~R. Rahman, \emph{Survey on topological methods in distributed computing}.

\bibitem[SS03]{SS03}
F.~W. Simmons and F.~E. Su, \emph{Consensus-halving via theorems of
  {B}orsuk-{U}lam and {T}ucker}, Math. Social Sci. \textbf{45} (2003), no.~1,
  15--25.

\bibitem[Su97]{Su97}
F.~E. Su, \emph{Borsuk-{U}lam implies {B}rouwer: a direct construction}, Amer.
  Math. Monthly \textbf{104} (1997), no.~9, 855--859.

\bibitem[Tuc46]{Tuc46}
A.~W. Tucker, \emph{Some topological properties of disk and sphere}, Proc.
  {F}irst {C}anadian {M}ath. {C}ongress, {M}ontreal, 1945, University of
  Toronto Press, 1946, pp.~285--309.

\bibitem[Yan09]{Yan09}
M.~Yannakakis, \emph{Equilibria, fixed points, and complexity classes},
  Computer Science Review \textbf{3} (2009), no.~2, 71--85.

\bibitem[{\v{Z}}iv10]{Ziv10}
Rade~T {\v{Z}}ivaljevi{\'c}, \emph{Oriented matroids and ky fan's theorem},
  Combinatorica \textbf{30} (2010), no.~4, 471--484.

\end{thebibliography}
}

\end{document}